\documentclass[runningheads,orivec]{llncs}
\usepackage[T1]{fontenc}
\usepackage{graphicx}
\usepackage{microtype}

\usepackage{amsmath, amssymb, enumerate, mathtools, hyperref, graphicx, float, xcolor, apptools, bm, bbm,quiver}
\usepackage[capitalise]{cleveref}

\newcommand{\R}{\mathbb R}

\newcommand{\D}{\mathcal D}
\newcommand{\Dev}{\mathrm D}

\newcommand{\E}{\mathbb E}

\newcommand{\inv}{^{-1}}

\usepackage[
    obeyFinal,
    textsize=footnotesize,
]{todonotes}

\newenvironment{proof}[1][Proof]{%
    \par\noindent\textit{#1. }% Proof heading (default: "Proof")
    \ignorespaces
}{%
    \hfill$\square$ % Right-aligned Q.E.D. symbol
    \par
}

\begin{document}
\title{Geometric Gaussian Approximations of Probability Distributions}
\author{Nathaël Da Costa\inst{1}%\orcidID{0000-1111-2222-3333}
\and
Bálint Mucsányi\inst{1}%\orcidID{1111-2222-3333-4444}
\and
Philipp Hennig\inst{1}}%\orcidID{2222--3333-4444-5555}}
\authorrunning{N. Da Costa et al.}
\institute{Tübingen AI Center, University of Tübingen, Germany
\email{nathael.da-costa@uni-tuebingen.de}}
\maketitle              % typeset the header of the contribution
\begin{abstract}
Approximating complex probability distributions, such as Bayesian posterior distributions, is of central interest in many applications. We study the expressivity of geometric Gaussian approximations. These consist of approximations by Gaussian pushforwards through diffeomorphisms or Riemannian exponential maps. We first review these two different kinds of geometric Gaussian approximations. Then we explore their relationship to one another. We further provide a constructive proof that such geometric Gaussian approximations are universal, in that they can capture any probability distribution. Finally, we discuss whether, given a family of probability distributions, a common diffeomorphism can be found to obtain uniformly high-quality geometric Gaussian approximations for that family.

\end{abstract}
\section{Introduction}
Approximating probability distributions is a central task of interest in statistics and machine learning. For example, density estimation aims to learn an underlying probability distribution given data sampled from that distribution. In Bayesian statistics, posterior distributions---that is, probability distributions over the parameters conditioned on data---can often not be computed in closed form and must instead be approximated.

Such applications thus require approximation schemes for probability distributions. For example, in Bayesian statistics, one may use a Gaussian to approximate a distribution with the help of a Laplace approximation or variational inference. Such Gaussian approximations are clearly limited in their expressivity, thus requiring more expressive approximating families. Instead, one may approximate the distribution of interest through a non-linear pushforward of a Gaussian distribution. For example, deep neural networks such as normalising flows or diffusion models are used to learn an approximate mapping from a Gaussian distribution to the distribution of interest \cite{kobyzev_normalizing_2021,yang_diffusion_2023}.

In this paper, we study reparametrised and Riemannian Gaussian approximations. The former consists of approximations to probability distributions by diffeomorphism pushforwards of Gaussians \cite{mackay_choice_1998,hobbhahn_laplace_2022,bui_likelihood_2024}, while the latter consists of approximations by exponential map pushforwards of Gaussians in a tangent space under some Riemannian metric \cite{bergamin_riemannian_nodate,yu_riemannian_2024}.

Our first result (\cref{thm:reparamga-to-riemannga,thm:converse}) shows that there is a one-to-one correspondence between reparametrised Gaussian approximations and Riemannian Gaussian approximations under flat metrics.

Our second result (\cref{thm:universality}) then shows that, under mild regularity assumptions, even reparametrised Gaussian approximations can approximate any probability distribution. This makes both reparametrised and Riemannian Gaussian approximations \textit{universal}.

Finally, we discuss whether, given a family of probability distributions, there exists a diffeomorphism such that every member of the family may be well approximated by a Gaussian pushforward through this diffeomorphism. An in\-for\-mation-geometric argument shows that it is generally not possible to do this exactly. This leaves several open questions as to the form of the optimal such diffeomorphism.

\section{Geometric Gaussian Approximations}
\begin{definition}\label{def:ga}
    A \textnormal{Gaussian Approximation} is an approximation to a probability measure $\mathrm p$ on $\R^d$ by a Gaussian distribution $\mathcal N(\bm\mu,\bm\Sigma)$.
\end{definition}

For example, given some dataset $\D$ and a parametric statistical model, Bayes\-ian statistics is faced with the problem of computing the posterior distribution over the parameters
\begin{equation}\label{eq: posterior}
    p(\bm\theta\mid\D) = \frac{p(\D\mid\bm\theta)p(\bm\theta)}{p(\D)}
\end{equation}
where $p(\bm\theta)$ is a prior distribution over the parameters, $p(\D\mid \bm\theta)$ is a likelihood, and $p(\D) = \int_{\R^d}p(\D\mid\bm\theta)p(\bm\theta)\;d\bm\theta$ is the evidence. This last integral is usually intractable; thus, the posterior \cref{eq: posterior} is often approximated. The classical Laplace approximation to a posterior with probability density function $p(\bm\theta\mid\D)$ assumed to be twice continuously differentiable in $\bm\theta\in\R^d$ is a Gaussian approximation computed by taking a second-order Taylor expansion of the negative logarithm of the unnormalised posterior density
\begin{equation}
    -\log p(\D\mid\bm\theta) -\log p(\bm\theta) \approx \frac{1}{2}(\bm\theta-\bm\theta^*)^\top \bm H(\bm\theta-\bm\theta^*)+\text{const}
\end{equation}
where $\bm\theta^*$ is a maximum a posteriori estimator and $H_{ij} := \frac{\partial^2 (-\log p(\D\mid\bm\theta) -\log p(\bm\theta))}{\partial \theta_i\partial\theta_j}$. The posterior distribution is then approximated by $\mathcal N(\bm\theta^*, \bm H^{-1})$. Note that when the posterior is Gaussian, this approximation is exact.

Clearly, such Gaussian approximations are limited in their expressivity and fail to capture complex probability distributions faithfully. For this reason, more flexible classes of approximations have recently been proposed (\cref{example}).
\begin{definition}
    A \textnormal{Reparametrised Gaussian Approximation (ReparamGA)} is an approximation to a probability measure $\mathrm p$ on $\R^d$ by a pushforward $\bm\varphi_*\mathcal N(\bm\mu,\bm\Sigma)$ of a Gaussian distribution $\mathcal N(\bm\mu,\bm\Sigma)$, where $\bm\varphi\colon\R^d\to \R^d$ is some diffeomorphism.
\end{definition}
\begin{figure}
    \centering
    \includegraphics%[width=0.5\linewidth]
    {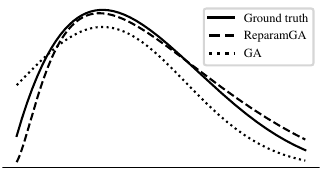}
    \caption{A Gaussian approximation to a distribution (specifically a Laplace approximation) and a reparametrised Gaussian approximation (specifically a reparametrised Laplace approximation).}
    \label{example}
\end{figure}
\begin{definition}
    A \textnormal{Riemannian Gaussian Approximation (RiemannGA)} is an approximation to a probability measure $\mathrm p$ on $\R^d$ by a pushforward $\bm\exp_{\bm\mu*}\mathcal N(\bm 0,\bm\Sigma)$ where $\mathcal N(\bm 0,\bm\Sigma)$ is a Gaussian distribution defined on a tangent space $T_{\bm\mu}\R^d$ at some $\bm\mu\in \R^d$, and $\bm\exp_{\bm\mu}\colon T_{\bm\mu}\R^d \to \R^d$ is the exponential map under some Riemannian metric $g$ on $\R^d$.
\end{definition}

\section{Relationship Between the Approximations}

We now show that a ReparamGA is a special case of a RiemannGA.

\begin{theorem}\label{thm:reparamga-to-riemannga}
    Any given ReparamGA $\bm\varphi_*\mathcal N(\bm\mu,\bm\Sigma)$ coincides with the RiemannGA $\bm\exp_{\bm\varphi(\bm\mu)*}\mathcal N\left(\bm 0, \Dev_{\bm\mu}\bm\varphi\bm\Sigma(\Dev_{\bm\mu}\bm\varphi)^\top\right)$ under the Riemannian metric $\bm\varphi^{-1*}g_e$, with $g_e$, the standard Euclidean metric on $\R^d$.
\end{theorem}
\begin{proof}
    $\bm\varphi$ is an isometry from $(\R^d,g_e)$ to $(\R^d,\bm\varphi^{-1*}g_e)$. Let $\bm\exp_{\bm\varphi(\bm\mu)}\colon T_{\bm\varphi(\bm\mu)}\R^d\to \R^d$ be the exponential map with respect to the metric $\bm\varphi^{-1*}g_e$.
    $g_e$ is a flat metric, hence so is $\bm\varphi^{-1*}g_e$. Thus, we deduce from \cite[Theorem 2.1]{noauthor_riemannian_nodate} that if we let $\tilde g_e$ be the Euclidean metric on $T_p\R^d$ induced by the inner product on this tangent space, $\bm\exp_{\bm\varphi(\bm\mu)}$ is an isometry from $(T_{\bm\varphi(\bm\mu)}\R^d, \tilde g_e)$ to $(\R^d, \bm\varphi^{-1*}g_e)$. Consequently, we have the commutative diagram of isometries
% https://q.uiver.app/#q=WzAsNCxbMCwxLCIgKFxcbWF0aGJiIFJeZCwgZ19lKSJdLFsxLDEsIiAoXFxtYXRoYmIgUl5kLCBcXGJtIFxcdmFycGhpXnstMSp9Z19lKSJdLFsxLDAsIihUX3tcXGJtXFx2YXJwaGkoXFxibVxcbXUpfVxcbWF0aGJiIFJeZCxcXHRpbGRlIGdfZSkiXSxbMCwwXSxbMCwxLCJcXGJtXFx2YXJwaGkiXSxbMiwxLCJcXGJtXFxleHBfe1xcYm1cXHZhcnBoaShcXG11KX0iXSxbMCwyLCJcXGJtXFxleHBfe1xcYm1cXHZhcnBoaShcXGJtXFxtdSl9XFxpbnZcXGNpcmNcXGJtXFx2YXJwaGkiLDAseyJsYWJlbF9wb3NpdGlvbiI6MH1dXQ==
\[\begin{tikzcd}
	{} & {(T_{\bm\varphi(\bm\mu)}\mathbb R^d,\tilde g_e)} \\
	{ (\mathbb R^d, g_e)} & { (\mathbb R^d, \bm \varphi^{-1*}g_e)}
	\arrow["{\bm\exp_{\bm\varphi(\mu)}}", from=1-2, to=2-2]
	\arrow["{\bm\exp_{\bm\varphi(\bm\mu)}\inv\circ\bm\varphi}"{pos=0.5}, from=2-1, to=1-2]
	\arrow["{\bm\varphi}", from=2-1, to=2-2]
\end{tikzcd}\]
    $\bm\exp_{\bm\varphi(\bm\mu)}\inv\circ\bm\varphi$  is thus an isometry from $(\R^d,g_e)$ to $(T_{\bm\varphi(\bm\mu)}\R^d, \tilde g_e)$. This is an isometry between Euclidean spaces which maps $\bm \mu$ to $\bm 0$, so it sends the Gaussian distributions on $\R^d$ centered at $\bm\mu$ to the Gaussian distributions on $T_{\bm\varphi(\bm\mu)}\R^d$ centered at $\bm 0$. Specifically, it sends $\mathcal N(\bm\mu,\bm\Sigma)$ to
    \begin{align*}
        &\mathcal N\left(\bm 0, \Dev_{\bm\mu}\left(\bm\exp_{\bm\varphi(\bm\mu)}\inv\circ\bm\varphi\right)\bm \Sigma \Dev_{\bm\mu}\left(\bm\exp_{\bm\varphi(\bm\mu)}\inv\circ\bm\varphi\right)^\top\right) \\
        &\hspace{4.5em}= \mathcal N\left(\bm 0, \Dev_{\bm\varphi(\bm\mu)}\bm\exp_{\bm\varphi(\bm\mu)}\inv\Dev_{\bm\mu}\bm\varphi\bm\Sigma(\Dev_{\bm\mu}\bm\varphi)^\top\left(\Dev_{\bm\varphi(\bm\mu)}\bm\exp_{\bm\varphi(\bm\mu)}\inv\right)^\top\right) \\
        &\hspace{4.5em}= \mathcal N\left(\bm 0, \Dev_{\bm\mu}\bm\varphi\bm\Sigma(\Dev_{\bm\mu}\bm\varphi)^\top\right)
    \end{align*}
    where the first equality holds by the chain rule and the second by the fact that the derivative of the exponential map at the origin is the identity. Here, $\Dev_{\bm x}f$ denotes the matrix representation of the derivative of $f$ at $\bm x$ with respect to the standard Euclidean coordinates, not the abstract linear map between tangent spaces.
\end{proof}

Gaussianity plays an important role in the proof of \cref{thm:reparamga-to-riemannga}, where it is used that Gaussian distributions are closed under affine transformations and thus isometries of Euclidean spaces.

In the proof of \cref{thm:reparamga-to-riemannga}, we see that when a Riemannian metric $g$ on $\R^d$ is flat, the exponential map $\bm\exp_{\bm\mu}$ is a diffeomorphism from $T_{\bm\mu}\R^d$ to $\R^d$. Thus we get the partial converse to \cref{thm:reparamga-to-riemannga}:
\begin{theorem}\label{thm:converse}
    A RiemannGA $\bm\exp_{\bm\mu*}\mathcal N(\bm 0,\bm\Sigma)$ under a flat Riemannian metric $g$ is a ReparamGA $\bm\varphi_*\mathcal N(\bm0,\bm\Sigma)$ with $\bm\varphi = \bm\exp_{\bm\mu}$.
\end{theorem}

\Cref{thm:converse} can only be a partial converse, as when $g$ is not flat, $\bm\exp_{\bm\mu}$ is not a diffeomorphism.

Thus, it may seem at this point that RiemannGA is more expressive than ReparamGA. The following section shows that this is incorrect, as ReparamGA can already express virtually any distribution.

\section{Universality of Reparametrised Gaussian Approximations}
In this section, we explicitly construct a diffeomorphism from the unit Gaussian distribution to any (sufficiently regular) probability distribution. The inverse of this transformation is known in the literature as the Rosenblatt transformation \cite{ditlevsen_structural_1996}. Therefore this construction is not novel, but we include it here for completeness.

We first prove a crucial lemma.
\begin{lemma}\label{lem:mapping-distributions}
    Let $\mathrm p^1$, $\mathrm p^2$ be probability measures on $\R^d$ with non-vanishing smooth probability density functions $f^1$, $f^2$ respectively. Then there exists a diffeomorphism $\bm\varphi\colon\R^d\to\R^d$ such that $\bm\varphi_*\mathrm p^1 = \mathrm p^2$.
\end{lemma}
\begin{proof}
    We show that there exist diffeomorphisms $\bm\varphi^i\colon \R^d\to (0,1)^d$ such that $\bm\varphi^i_*\mathrm p^i$ are both  the uniform measure on $(0,1)^d$ for $i=1,2$. Then the result follows by taking $\bm\varphi := (\bm\varphi^2)\inv\circ\bm \varphi^1$.
    
    Let $\bm X$ be an $\R^d$ valued random variable with law $\mu^1$. Let $F_{X_1}$ be the c.d.f.~of the first component of $\bm X$. $f^1$ being smooth and non-vanishing implies that $F_{X_1}\colon \R\to (0,1)$ is a diffeomorphism. Define $\varphi^1_1(\bm x) := F_{X_1}(x_1)$. For $i>1$, let $F_{X_i|X_1=x_1,\dots,X_{i-1}=x_{i-1}}$ be the conditional c.d.f.~of the $i$\textsuperscript{th} component of $\bm X$, i.e.,
    \[
        F_{X_i|X_1=x_1,\dots,X_{i-1}=x_{i-1}}(x_i) := \frac{\int_{-\infty}^{x_i}\int_{\R^{d-i}}f^1(x_1,\dots,x_{i-1}, y_i,\dots, y_d)\;dy_d\cdots dy_i}{\int_{\R^{d-i+1}}f^1(x_1,\dots,x_{i-1}, y_i,\dots, y_d)\;dy_d\cdots dy_i},
    \]
    which is a diffeomorphism $\R\to(0,1)$. Define $\varphi^1_i(\bm x) := F_{X_i|X_1=x_1,\dots,X_{i-1}=x_{i-1}}(x_i)$.

    Now for $u_1,\dots,u_d\in(0,1)$,
    \begin{align*}
        P&(\varphi_1^1(\bm X) \leq u_1,\dots,\varphi_d^1(\bm X) \leq u_d) \\
        &= P\left(X_1\leq F\inv_{X_1}(u_1),\dots,X_d\leq F\inv_{X_d|X_1,\dots,X_{d-1}}(u_d)\right) \\
        &= \int_{-\infty}^{F\inv_{X_1}(u_1)}\dots\int_{-\infty}^{F\inv_{X_d|X_1=x_1,\dots,X_{d-1}=x_{d-1}}(u_d)}f^1(x_1,\dots,x_d)\;dx_d\cdots dx_1 \\
        &= \int_{-\infty}^{F\inv_{X_1}(u_1)}\dots\int_{-\infty}^{F\inv_{X_{d-1}|X_1=x_1,\dots,X_{d-2}=x_{d-2}}(u_d)}\int_\R f^1(x_1,\dots,x_d)\;dx_d\\
        &\quad F_{X_d|X_1=x_1,\dots, X_{d-1}=x_{d-1}}\left(F\inv_{X_d|X_1=x_1,\dots, X_{d-1}=x_{d-1}}(u_d)\right)\;dx_{d-1} \cdots dx_1 \\
        &= \int_{-\infty}^{F\inv_{X_1}(u_1)}\dots\int_{-\infty}^{F\inv_{X_{d-1}|X_1=x_1,\dots,X_{d-2}=x_{d-2}}(u_d)}\int_\R f^1(x_1,\dots,x_d)\;dx_d\cdots dx_1 \cdot u_d \\
        &= \dots \\
        &= \int_{-\infty}^{F_{X_1}\inv(u_1)}\int_{\R^{d-1}}f^1(x_1,\dots,x_d)\;dx_d\cdots dx_1 \cdot u_2\cdots u_d \\
        &= F_{X_1}\left(F_{X_1}\inv(u_1)\right)u_2\cdots u_d \\
        &= u_1\cdots u_d
        % &= \int_{\R^d}\;dx_d F_{X_d|X_1=x_1,\dots, X_{d-1}=x_{d-1}}\left(F\inv_{X_d|X_1=x_1,\dots, X_{d-1}=x_{d-1}}(u_d)\right)\; d_x_{d-1}\cdots dx_1 F_{X_1}
    \end{align*}
    which shows that $\bm\varphi^1_*\mu^1$ is the uniform measure on $(0,1)^d$.
\end{proof}

From \cref{lem:mapping-distributions}, we directly obtain the universality of ReparamGA for probability measures with non-vanishing densities:
\begin{theorem}\label{thm:universality}
    For $\mathrm p$ a probability measure on $\R^d$ with a non-vanishing smooth probability density function, there exists a diffeomorphism $\bm \varphi\colon \R^d \to \R^d$ such that $\mathrm p = \bm\varphi_* \mathcal N(\bm 0, \bm I)$. That is, there is a ReparamGA which is exact. As a consequence, there is a RiemannGA which is exact.
\end{theorem}
\begin{proof}
    Take $\mathrm p_1 = \mathcal N(\bm 0,\bm I)$, $\mathrm p_2 = \mathrm p$ in \cref{lem:mapping-distributions}. The existence of an exact RiemannGA then follows from \cref{thm:reparamga-to-riemannga}.
\end{proof}

Note that, from the proof of \cref{lem:mapping-distributions}, if we only require the probability density function to be $C^k$, then we still get \cref{lem:mapping-distributions} and hence \cref{thm:universality} with a $C^{k+1}$ diffeomorphism $\bm\varphi$.

The transformation $\bm\varphi$ of \cref{thm:universality} is not unique, as pre-composing $\bm\varphi$ with any orthogonal transformation will still yield a mapping from $\mathcal N(\bm 0,\bm I)$ to $\mathrm p$.

Finally, note that, unlike in \cref{thm:reparamga-to-riemannga}, \cref{thm:universality} does not make use of Gaussianity, and Gaussians could have been replaced with other approximating probability distributions.

\section{Families of Reparametrised Gaussian Approximations}\label{sec:families}
Now suppose that instead of approximating one target probability measure $\mathrm p$ and choosing a $\bm\varphi$ to perform a good (or exact) ReparamGA, we are tasked with choosing a fixed $\bm\varphi$ that creates uniformly good ReparamGAs over some family of probability measures $\{\mathrm p_\alpha\}_{\alpha}$. This is useful if we know the distribution of interest belongs to such a family \cite{mackay_choice_1998,mucsanyi_rethinking_2025}.

One may wonder if a similarly strong result as \cref{thm:universality} can be obtained in this more general case. That is, can we find a diffeomorphism $\bm\varphi\colon\R^d\to\R^d$ such that $\mathrm p_{\alpha} = \bm\varphi_*\mathcal N(\bm \mu_{\alpha}, \bm \Sigma_{\alpha})$ for all $\alpha$ with some $\bm\mu_{\alpha}$, $\bm\Sigma_{\alpha}$?

This is generally not possible. Indeed, if we assume $\{\mathrm p_\alpha\}_{\alpha}$ is a statistical manifold $\mathcal M$, and write $\mathcal G$ for the statistical manifold of Gaussian distributions over $\R^d$, then, by Chentsov's theorem \cite[Section 5.1.4]{ay_information_2017}, any diffeomorphism $\bm\varphi\colon \R^d\to \R^d$ preserves the Fisher information metric on $\mathcal G$. That is, $\bm\varphi$ is an isometry from $(\mathcal G,g_F)$ to $(\bm\varphi_*\mathcal G, g_F)$ where $g_F$ denotes the respective Fisher information metrics. Thus, $\mathcal G$ can be mapped surjectively onto $\mathcal M$ only if $(\mathcal G,g_F)$ is isometric to a submanifold of $(\mathcal M,g_F)$. This is clearly restrictive.

When this is not possible, one may ask how to choose a $\bm\varphi$ such that $\bm\varphi^* \mathcal G$ is `close' to $\mathcal M$, such that probability measures in $\mathcal M$ may be well approximated by Gaussians. Defining a good notion of `closeness', and a corresponding good $\bm\varphi$, is an interesting direction of future research.

One option proposed in \cite[Appendix D1]{mucsanyi_rethinking_2025} is to choose probability measures $\mathrm p_1 \in \mathcal G$, $\mathrm p_2 \in \mathcal M$ and to use \cref{thm:universality} to construct a $\bm\varphi$ such that $\bm\varphi_*\mathrm p_1 = \mathrm p_2$. One would hope that such a $\bm\varphi$ would make $\bm\varphi^* \mathcal G$ and $\mathcal M$ `close'. More generally, given a (second-order) probability distribution $Q$ on $\mathcal M$ and some statistical divergence $\mathrm D$, one could attempt to find a diffeomorphism $\bm\varphi$ that minimises the expected divergence $\E_{\mathrm p_2\sim Q}[\inf_{\mathrm p_1\in \mathcal G}\mathrm D(\mathrm p_2,\bm\varphi_*\mathrm p_1)]$. The suggestion of $\bm\varphi$ that makes $\bm\varphi_*\mathcal{G}$ and $\mathcal{M}$ intersect at a select point $\bm\varphi_*\mathrm p_1 = \mathrm p_2$  comes from choosing $Q$ to be a Dirac measure at some $\mathrm p_2\in \mathcal M$.

\subsection{Visualisation of Statistical Manifolds}
In \cref{fig}, $\bm\varphi_*\mathcal G$ is the statistical manifold of one-dimensional Gaussian pushforwards by some diffeomorphism $\bm \varphi$, and $\mathcal M$ is some other statistical manifold. Since $\mathcal G$ and $\mathcal M$ have non-isometric Fisher metrics, they cannot be mapped surjectively to one another. Instead, $\bm\varphi$ is chosen such that $\bm\varphi_*\mathcal G$ and $\mathcal M$ intersect at a point. Intuitively, one would expect this to make $\bm\varphi_*\mathcal G$ and $\mathcal M$ `close' together.

Such a sketch is limited, as while the statistical 2-dimensional manifolds live in a common infinite-dimensional space of probability distributions over some space, there is generally no 3-dimensional slice containing both of them, preventing such a visualisation. Moreover, even one such manifold may be impossible to visualise isometrically in 3 dimensions: here, to visualise the hyperbolic geometry of $\mathcal G$ \cite[Example 3.1]{ay_information_2017}, we use the hyperboloid model, which is not an isometric embedding of $\mathcal G$ in Euclidean space but in Minkowski space.
\begin{figure}
    \centering
    \includegraphics[width=0.5\linewidth,trim={0cm 0.9cm 0 0.4cm}]
    {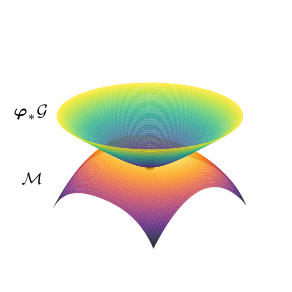}
    \caption{Sketch of statistical manifolds. $\mathcal M$ is some statistical manifold, and $\bm\varphi_*\mathcal G$ is the statistical manifold of one-dimensional Gaussian pushforwards by a diffeomorphism $\bm \varphi$ making the manifolds intersect at some point.}
    \label{fig}
\end{figure}
\section{Conclusion}
    We reviewed two different types of approximations from the literature under a common umbrella of geometric Gaussian approximations, clarifying their exact relationship to one another. We then presented the construction of the Rosenblatt transformation, which proves that such geometric Gaussian approximations are universal. Future work could develop a theory for addressing the questions raised in \cref{sec:families} about the approximate mapping of statistical manifolds to one another through diffeomorphisms of the base measure space. These appear to touch on some deep questions in information geometry.
    
    A third type of geometric Gaussian approximation, proposed in \cite{roy_reparameterization_2024} and based on Riemannian diffusions, remains to be included in this framework.
\begin{credits}
\subsubsection{\ackname}
\end{credits}
\bibliographystyle{splncs04}
\bibliography{bibliography}

\end{document}